\documentclass[10pt]{easychair}

\usepackage[utf8x]{inputenc}
\usepackage[english]{babel}
\usepackage{amsmath,amsfonts,amssymb,amsthm,shuffle, mathtools}
\usepackage{lmodern}

\usepackage{xcolor}

\definecolor{ColBlack}{RGB}{0,0,0} 
\definecolor{ColWhite}{RGB}{255,255,255} 
\definecolor{Col1}{RGB}{133,6,6} 
\definecolor{Col2}{RGB}{198,8,0} 
\definecolor{Col3}{RGB}{174,74,52} 
\definecolor{Col4}{RGB}{103,113,121} 
\definecolor{Col5}{RGB}{90,94,107} 
\definecolor{Col6}{RGB}{70,63,50} 

\usepackage{tikz}
\usetikzlibrary{shapes}
\usetikzlibrary{fit}
\usetikzlibrary{decorations.pathmorphing}

\usepackage{mathtools}
\usepackage{dsfont}
\usepackage{wasysym}
\usepackage{stmaryrd}
\usepackage{cite}
\usepackage{subfig}
\usepackage{multirow}
\usepackage{enumitem}
\usepackage{multicol}

\allowdisplaybreaks

\numberwithin{equation}{section}

\makeatletter
\def\l@section{\@tocline{1}{3pt}{1pc}{5pc}{}}
\def\l@subsection{\@tocline{2}{2pt}{2pc}{5pc}{}}
\makeatother

\newtheorem{Theorem}{Theorem}[section]
\newtheorem{Proposition}[Theorem]{Proposition}
\newtheorem{Lemma}[Theorem]{Lemma}

\renewcommand{\leq}{\leqslant}
\renewcommand{\geq}{\geqslant}

\title{Generalizations of the associative operad \\
and convergent rewrite systems}
\titlerunning{Generalizations of the associative operad and convergent
rewrite systems}
\date{\today \\}
\author{Cyrille Chenavier\inst{1}\thanks{\tt cyrille.chenavier@u-pem.fr}
\and Christophe Cordero\inst{1}\thanks{\tt christophe.cordero@u-pem.fr}
\and Samuele Giraudo\inst{1}\thanks{\tt samuele.giraudo@u-pem.fr}}
\authorrunning{Chenavier, Cordero, Giraudo}
\institute{Université Paris-Est, LIGM (UMR $8049$), CNRS, ENPC, ESIEE
Paris, UPEM, F-$77454$, Marne-la-Vallée, France}

\newcommand{\Hide}[1]{\textcolor{Col4}{\tt [hidden]}}

\newcommand{\Def}[1]{\textcolor{Col3}{\em #1}}

\tikzstyle{Centering}=[{baseline={([yshift=-0.5ex]current
    bounding box.center)}}]

\newcommand{\Oca}{\mathcal{O}}
\newcommand{\Rfr}{\mathfrak{r}}
\newcommand{\Sfr}{\mathfrak{s}}
\newcommand{\Tfr}{\mathfrak{t}}

\newcommand{\Zero}{\mathtt{0}}
\newcommand{\Two}{\mathtt{2}}
\newcommand{\HilbertSeries}{\mathcal{H}}

\newcommand{\Mag}{\mathsf{Mag}}
\newcommand{\As}{\mathsf{As}}
\newcommand{\CAs}[1]{\mathsf{CAs}^{(#1)}}
\newcommand{\PrefixWord}{\mathrm{p}}
\newcommand{\Deg}{\mathrm{deg}}

\DeclareMathOperator{\Product}{\star}
\DeclareMathOperator{\Congr}{\equiv}
\DeclareMathOperator{\Rew}{\to}

\DeclareMathOperator{\RewContext}{\Rightarrow}

\newcommand{\CongrCAs}[1]{\Congr^{(#1)}}

\tikzstyle{Node}=[circle,draw=Col1!80,fill=Col1!8,inner sep=1pt,
minimum size=2mm,thick,font=\scriptsize]
\tikzstyle{Edge}=[draw=Col2!80,cap=round,thick]
\tikzstyle{Leaf}=[rectangle,draw=ColBlack!70,fill=ColBlack!16,
inner sep=0pt,minimum size=1mm,thick]
\tikzstyle{NodeST}=[font=\footnotesize]

\begin{document}

\maketitle

\vspace{-1em}

\begin{abstract}
    The associative operad is the quotient of the magmatic operad by
    the operad congruence identifying the two binary trees of degree
    $2$. We introduce here a generalization of the associative operad
    depending on a nonnegative integer $d$, called $d$-comb associative
    operad, as the quotient of the magmatic operad by the operad
    congruence identifying the left and the right comb binary trees of
    degree $d$. We study the case $d = 3$ and provide an orientation
    of its space of relations by using rewrite systems on trees and
    the Buchberger algorithm for operads to obtain a convergent
    rewrite system.
\end{abstract}

\section*{Introduction}
Associative algebras are spaces endowed with a binary product $\Product$
satisfying among others the associativity law
\begin{math}
    (x_1 \Product x_2) \Product x_3 = x_1 \Product (x_2 \Product x_3)
\end{math}.
It is well-known that the associative algebras are representations of 
the associative (nonsymmetric) operad $\As$. This operad can be seen as 
the quotient of the magmatic operad $\Mag$ (the free operad of binary
trees on the binary generator~$\Product$) by the operad congruence
$\Congr$ satisfying
\begin{equation} \label{equ:congruence_as}
    \begin{tikzpicture}[xscale=.24,yscale=.24,Centering]
        \node(0)at(0.00,-3.33){};
        \node(2)at(2.00,-3.33){};
        \node(4)at(4.00,-1.67){};
        \node[NodeST](1)at(1.00,-1.67)
            {\begin{math}\Product\end{math}};
        \node[NodeST](3)at(3.00,0.00)
            {\begin{math}\Product\end{math}};
        \draw[Edge](0)--(1);
        \draw[Edge](1)--(3);
        \draw[Edge](2)--(1);
        \draw[Edge](4)--(3);
        \node(r)at(3.00,1.5){};
        \draw[Edge](r)--(3);
    \end{tikzpicture}
    \Congr
    \begin{tikzpicture}[xscale=.24,yscale=.24,Centering]
        \node(0)at(0.00,-1.67){};
        \node(2)at(2.00,-3.33){};
        \node(4)at(4.00,-3.33){};
        \node[NodeST](1)at(1.00,0.00)
                {\begin{math}\Product\end{math}};
        \node[NodeST](3)at(3.00,-1.67)
                {\begin{math}\Product\end{math}};
        \draw[Edge](0)--(1);
        \draw[Edge](2)--(3);
        \draw[Edge](3)--(1);
        \draw[Edge](4)--(3);
        \node(r)at(1.00,1.5){};
        \draw[Edge](r)--(1);
    \end{tikzpicture}\,.
\end{equation}
These two binary trees are the syntax trees of the expressions appearing
in the above associativity law.

In a more combinatorial context and regardless of the theory of operads,
the Tamari order is a partial order on the set of the binary trees
having a fixed number of internal nodes $d$. This order is generated by
the covering relation consisting in rewriting a tree $\Tfr$ into a tree
$\Tfr'$ by replacing a subtree of $\Tfr$ of the form of the left member
of~\eqref{equ:congruence_as} into a tree of the form of the right member
of~\eqref{equ:congruence_as}. This transformation is known in a computer
science context as the right rotation operation~\cite{Knu98} and
intervenes in algorithms involving binary search trees~\cite{AVL62}. The
partial order hence generated by the right rotation operation is known
as the Tamari order~\cite{Tam62} and has a lot of combinatorial and
algebraic properties (see for instance~\cite{HT72,Cha06}).

A first connection between the associative operad and the Tamari order
is based upon the fact that the orientation of~\eqref{equ:congruence_as}
from left to right provides a convergent orientation (a terminating and 
confluent rewrite relation) of the congruence $\Congr$. The normal 
forms of the rewrite relation induced by the rewrite rule obtained by 
orienting~\eqref{equ:congruence_as} from left to right are right comb 
binary trees and are hence in one-to-one correspondence with the 
elements of~$\As$.

This work is intended to be a first strike in the study of the eventual
links between the Tamari order and some quotients of the operad $\Mag$.
In the long run, we would like to study quotients $\Mag/_{\Congr}$ of
$\Mag$ where $\Congr$ is an operad congruence generated by equivalence
classes of trees of a fixed degree. In particular, we would like to know
if $\Congr$ is generated by equivalence classes of trees forming
intervals of the Tamari order leads to algebraic properties for
$\Mag/_{\Congr}$ (like the description of orientations of its space of
relations, nice bases and Hilbert series).

We focus here on one of these quotients $\CAs{3}$ which is the operad
describing the category of the algebras equipped with a binary product
$\Product$ and subjected to the relation
\begin{math}
    ((x_1 \Product x_2) \Product x_3) \Product x_4
    =
    x_1 \Product (x_2 \Product (x_3 \Product x_4))
\end{math}.
This is a kind of associativity law in higher degree $d = 3$. This
operad is generated by an equivalence class of trees which is not an
interval for the Tamari order. As preliminary computer experiments show,
$\CAs{3}$ has oscillating first dimensions
(see~\eqref{equ:dimensions_CAs_3}), what is rather unusual among all
known operads. In this paper, we provide an orientation of the space of
relations of $\CAs{3}$. For this, we use rewrite systems on
trees~\cite{BN98} and the Buchberger algorithm for operads~\cite{DK10}.

This text is presented as follows. Section~\ref{sec:operad_Mag} contains
preliminaries about the magmatic operad and rewrite relations on trees.
In Section~\ref{sec:CAs_d}, we define the operad $\CAs{3}$ as a
particular case of a more general construction of generalizations
$\CAs{d}$, $d \geq 1$, of $\As$. Finally, Section~\ref{sec:CAs_3}
contains the orientation of the space of relations of
$\CAs{3}$ (Theorem~\ref{thm:convergent_rewrite_rule_CAs_3}). As
consequences, we obtain for $\CAs{3}$
the description of one of its Poincaré-Birkhoff-Witt bases
(Proposition~\ref{prop:PBW_basis_CAs_3}) and the description of
its Hilbert series (Proposition~\ref{prop:Hilbert_series_CAs_3}).

\section{The magmatic operad, quotients, and rewrite relations}
\label{sec:operad_Mag}
We consider nonsymmetric set-theoretic operads. Let $\Oca$ be such an
operad. We denote respectively by $\circ_i$ and $\circ$ the partial and
complete compositions of $\Oca$. For any $n \geq 1$, $\Oca(n)$ is the
set of the elements $x$ of $\Oca$ of arity $|x| = n$. We denote
by $\Mag$ the magmatic operad, that is the free operad over one binary
generator $\Product$, and we represent the elements of $\Mag$ by binary
trees. The \Def{arity} $|\Tfr|$ (resp. \Def{degree} $\Deg(\Tfr)$) of a
binary tree $\Tfr$ is its number of leaves (resp. internal nodes).
Given a binary tree $\Tfr$, we denote by $ \PrefixWord(\Tfr)$ the
\Def{prefix word} of $\Tfr$, that is the word on $\{\Zero, \Two\}$
obtained by a left to right depth-first traversal of $\Tfr$ and by
writing $\Zero$ (resp. $\Two$) when a leaf (resp. an internal node) is
encountered. The set of all words on $\{\Zero, \Two\}$ is endowed with
the lexicographic order $\leq$ induced by $\Zero < \Two$.

If $\Rew$ is a rewrite rule on $\Mag$ such that $\Sfr \Rew \Sfr'$ 
implies $|\Sfr| = |\Sfr'|$, we denote by $\RewContext$ the
\Def{rewrite relation induced} by $\Rew$. Formally we have
\begin{math}
    \Tfr\circ_i\left(\Sfr\circ\left[\Rfr_1,\dots,\Rfr_{n}\right]\right)
    \RewContext
    \Tfr\circ_i\left(\Sfr'\circ\left[\Rfr_1,\dots,\Rfr_{n}\right]\right)
\end{math},
if $\Sfr \Rew \Sfr'$ where $n = |\Sfr|$, and $\Tfr$, $\Rfr_1$, \dots,
$\Rfr_n$ are binary trees. In other words, one has
$\Tfr \RewContext \Tfr'$ if it is possible to obtain $\Tfr'$ from $\Tfr$
by replacing a subtree $\Sfr$ of $\Tfr$ by $\Sfr'$ whenever
$\Sfr \Rew \Sfr'$. We use here the standard terminology
(\Def{terminating}, \Def{confluent}, \Def{convergent}, \Def{branching
pair}, \Def{joinable}, \Def{normal form}, {\em etc.}) about rewrite
relations and rewrite systems~\cite{BN98}.

Given an operad $\Oca \simeq \Mag/_{\Congr}$ where $\Congr$ is an operad
congruence of $\Mag$, we say that $\Rew$ is an \Def{orientation} of
$\Congr$ if the reflexive, transitive, and symmetric closure of
$\RewContext$ is $\Congr$. We say that $\Rew$ is a \Def{convergent
orientation} if $\RewContext$ is convergent. When $\Rew$ is a convergent
orientation of $\Congr$, the set of all normal forms of $\RewContext$ is
a \Def{Poincaré-Birkhoff-Witt basis} of the operad $\Oca$ and its
elements are exactly the binary trees avoiding, as subtrees, the trees
appearing as left members in $\Rew$.

We shall use the following criterion to prove that a rewrite relation on
$\Mag$ is terminating.

\begin{Lemma} \label{lem:prefix_word_termination}
    Let $\Rew$ be a rewrite rule on $\Mag$. If for any
    $\Tfr, \Tfr' \in \Mag$ such that $\Tfr \Rew \Tfr'$ one has
    $\PrefixWord(\Tfr) > \PrefixWord(\Tfr')$, then the rewrite relation
    induced by $\Rew$ is terminating.
\end{Lemma}

Moreover, we shall use the following result appearing in~\cite{Gir16}
specialized on rewrite relation on $\Mag$ to prove that a terminating
rewrite relation is convergent.

\begin{Lemma} \label{lem:degree_confluence}
    Let $\Rew$ be a rewrite rule on $\Mag$ wherein all trees $\Tfr$ and 
    $\Tfr'$ such that $\Tfr \Rew \Tfr'$ have degrees at most $\ell$. 
    Then, if the rewrite relation $\RewContext$ induced by $\Rew$ is
    terminating and all its branching pairs of degrees at most
    $2\ell - 1$ are joinable, $\RewContext$ is convergent.
\end{Lemma}

\section{Generalizations of the associative operad} \label{sec:CAs_d}
It is known that the rewrite rule $\Rew$
orienting~\eqref{equ:congruence_as} from left to right is a convergent
orientation of (\ref{equ:congruence_as}). Then, a Poincaré-Birkhoff-Witt
basis of $\As$ is the set of all right comb binary trees.

Let us now define for any $d \geq 1$ the \Def{$d$-comb associative
operad} $\CAs{d}$ as the quotient operad $\Mag/_{\CongrCAs{d}}$ where
$\CongrCAs{d}$ is the smallest operad congruence of $\Mag$
satisfying
\begin{equation} \label{equ:congruence_Asd}
    \underbrace{(\dots (\Product \circ_1 \Product) \circ_1 \dots)
        \circ_1 \Product}
    _{d \mbox{ \footnotesize operands}}
    \enspace \CongrCAs{d} \enspace
    \underbrace{\Product \circ_2
        (\dots \circ_2 (\Product \circ_2 \Product) \dots)}
    _{d \mbox{ \footnotesize operands}}.
\end{equation}
In words, \eqref{equ:congruence_Asd} says that the left and the right
comb binary trees of degree $d$ are equivalent for $\CongrCAs{d}$.
Notice that $\CongrCAs{1}$ is trivial so that $\CAs{1} = \Mag$ and that
$\CongrCAs{2}$ is the operad congruence defined
by~\eqref{equ:congruence_as} so that $\CAs{2} = \As$.

As shown by the following statement, the operads $\CAs{d}$ are related
to each other.

\begin{Proposition} \label{prop:quotients_CAs_d}
    For any $d \geq 3$, $\CAs{d}$ is a quotient operad of $\CAs{2d - 1}$
    and $\CAs{2}$ is a quotient operad of $\CAs{d}$.
\end{Proposition}
\begin{proof}
    Since
    \begin{equation} \label{equ:quotients_CAs_d}
        \begin{tikzpicture}[xscale=.23,yscale=.25,Centering]
            \node[rotate = 45,font=\tiny] at (1,-5.5)
                {\begin{math}d - 1\end{math}};
            \draw[rotate = 45,font=\tiny] (1-5.8,0.75-6.2)
                arc(180:90:.1) -- (2.5-5.8, 0.85-6.2) arc(-90:0:.1);
            \draw[rotate = 45] (4.2-5.8,0.75-6.2)
                arc(0:90:.1) -- (2.7-5.8, 0.85-6.2) arc(270:180:.1);
            \node[rotate = 45,font=\tiny] at (7,0)
                {\begin{math}d - 1\end{math}};
            \draw[rotate = 45] (1+2.4,0.75-6.5)
                arc(180:90:.1) -- (2.5+2.4, 0.85-6.5) arc(-90:0:.1);
            \draw[rotate = 45] (4.2+2.4,0.75-6.5)
                arc(0:90:.1) -- (2.7+2.4, 0.85-6.5) arc(270:180:.1);
            \node(0)at(0.00,-9.17){};
            \node(10)at(10.00,-1.83){};
            \node(2)at(2.00,-9.17){};
            \node(4)at(4.00,-7.33){};
            \node(6)at(6.00,-5.50){};
            \node(8)at(8.00,-3.67){};
            \node[NodeST](1)at(1.00,-7.33)
                {\begin{math}\Product\end{math}};
            \node[NodeST](3)at(3.00,-5.50)
                {\begin{math}\Product\end{math}};
            \node[NodeST](5)at(5.00,-3.67)
                {\begin{math}\Product\end{math}};
            \node[NodeST](7)at(7.00,-1.83)
                {\begin{math}\Product\end{math}};
            \node[NodeST](9)at(9.00,0.00)
                {\begin{math}\Product\end{math}};
            \draw[Edge](0)--(1);
            \draw[Edge,dotted](1)--(3);
            \draw[Edge](10)--(9);
            \draw[Edge](2)--(1);
            \draw[Edge](3)--(5);
            \draw[Edge](4)--(3);
            \draw[Edge](5)--(7);
            \draw[Edge](6)--(5);
            \draw[Edge,dotted](7)--(9);
            \draw[Edge](8)--(7);
            \node(r)at(9.00,1.38){};
            \draw[Edge](r)--(9);
        \end{tikzpicture}
        \enspace \CongrCAs{d} \enspace
        \begin{tikzpicture}[xscale=.21,yscale=.18,Centering]
            \node[rotate = -55,font=\tiny] at (9.2,-3.2)
                {\begin{math}d - 1\end{math}};
            \draw[rotate = -55] (0.5+5.5,0.75+4)
                arc(180:90:.1) -- (2.25+5.5, 0.85+4) arc(-90:0:.1);
            \draw[rotate = -55] (4.2+5.5,0.75+4)
                arc(0:90:.1) -- (2.45+5.5, 0.85+4) arc(270:180:.1);
            \node[rotate = 55,font=\tiny] at (0.8,-3)
                {\begin{math}d - 1\end{math}};
            \draw[rotate = 55] (0.5-4.5,0.75-4)
                arc(180:90:.1) -- (2.25-4.5, 0.85-4) arc(-90:0:.1);
            \draw[rotate = 55] (4.2-4.5,0.75-4)
                arc(0:90:.1) -- (2.45-4.5, 0.85-4) arc(270:180:.1);
            \node(0)at(0.00,-8.25){};
            \node(10)at(10.00,-8.25){};
            \node(2)at(2.00,-8.25){};
            \node(4)at(4.00,-5.50){};
            \node(6)at(6.00,-5.50){};
            \node(8)at(8.00,-8.25){};
            \node[NodeST](1)at(1.00,-5.50)
                {\begin{math}\Product\end{math}};
            \node[NodeST](3)at(3.00,-2.75)
                {\begin{math}\Product\end{math}};
            \node[NodeST](5)at(5.00,0.00)
                {\begin{math}\Product\end{math}};
            \node[NodeST](7)at(7.00,-2.75)
                {\begin{math}\Product\end{math}};
            \node[NodeST](9)at(9.00,-5.50)
                {\begin{math}\Product\end{math}};
            \draw[Edge](0)--(1);
            \draw[Edge,dotted](1)--(3);
            \draw[Edge](10)--(9);
            \draw[Edge](2)--(1);
            \draw[Edge](3)--(5);
            \draw[Edge](4)--(3);
            \draw[Edge](6)--(7);
            \draw[Edge](7)--(5);
            \draw[Edge](8)--(9);
            \draw[Edge,dotted](9)--(7);
            \node(r)at(5.00,2.06){};
            \draw[Edge](r)--(5);
        \end{tikzpicture}
        \enspace \CongrCAs{d} \enspace
        \begin{tikzpicture}[xscale=.23,yscale=.25,Centering]
            \node[rotate = -45,font=\tiny] at (9,-5.5)
                {\begin{math}d - 1\end{math}};
            \draw[rotate = -45] (1-0.5,0.75 + 0.5)
                arc(180:90:.1) -- (2.5-0.5, 0.85 + 0.5) arc(-90:0:.1);
            \draw[rotate = -45] (4.2-0.5,0.75 + 0.5)
                arc(0:90:.1) -- (2.7-0.5, 0.85 + 0.5) arc(270:180:.1);
            \node[rotate = -45,font=\tiny] at (3,0)
                {\begin{math}d - 1\end{math}};
            \draw[rotate = -45] (1+7.5,0.75+0.8)
                arc(180:90:.1) -- (2.5+7.5, 0.85+0.8) arc(-90:0:.1);
            \draw[rotate = -45] (4.2+7.5,0.75+0.8)
                arc(0:90:.1) -- (2.7+7.5, 0.85+0.8) arc(270:180:.1);
            \node(0)at(0.00,-1.83){};
            \node(10)at(10.00,-9.17){};
            \node(2)at(2.00,-3.67){};
            \node(4)at(4.00,-5.50){};
            \node(6)at(6.00,-7.33){};
            \node(8)at(8.00,-9.17){};
            \node[NodeST](1)at(1.00,0.00)
                {\begin{math}\Product\end{math}};
            \node[NodeST](3)at(3.00,-1.83)
                {\begin{math}\Product\end{math}};
            \node[NodeST](5)at(5.00,-3.67)
                {\begin{math}\Product\end{math}};
            \node[NodeST](7)at(7.00,-5.50)
                {\begin{math}\Product\end{math}};
            \node[NodeST](9)at(9.00,-7.33)
                {\begin{math}\Product\end{math}};
            \draw[Edge](0)--(1);
            \draw[Edge](10)--(9);
            \draw[Edge](2)--(3);
            \draw[Edge,dotted](3)--(1);
            \draw[Edge](4)--(5);
            \draw[Edge](5)--(3);
            \draw[Edge](6)--(7);
            \draw[Edge](7)--(5);
            \draw[Edge](8)--(9);
            \draw[Edge,dotted](9)--(7);
            \node(r)at(1.00,1.38){};
            \draw[Edge](r)--(1);
        \end{tikzpicture}
    \end{equation}
    where a dotted edge between two internal nodes denotes a left or a
    right comb tree of degree $d - 1$ (hence, the trees
    of~\eqref{equ:quotients_CAs_d} are of degree $2d - 1$), the relation
    $\Tfr \CongrCAs{2d - 1}\Tfr'$ implies $\Tfr \CongrCAs{d} \Tfr'$ for
    any trees $\Tfr$ and $\Tfr'$. Hence, $\CongrCAs{2d - 1}$ is finer
    than $\CongrCAs{d}$, whence the first part of the statement of the
    proposition. The second part of the statement of the proposition is
    a consequence of the fact that the relation
    $\Tfr \CongrCAs{d} \Tfr'$ implies $\Tfr \CongrCAs{2} \Tfr'$ for any
    trees $\Tfr$ and $\Tfr'$.
\end{proof}

\section{The \texorpdfstring{$3$}{3}-comb associative operad}
\label{sec:CAs_3}
We now focus on the study of the operad $\CAs{3}$. By definition, this
operad is the quotient of $\Mag$ by the operad congruence spanned by the
relation
\begin{equation} \label{equ:rew_1}
    \begin{tikzpicture}[xscale=.22,yscale=.23,Centering]
        \node(0)at(0.00,-5.25){};
        \node(2)at(2.00,-5.25){};
        \node(4)at(4.00,-3.50){};
        \node(6)at(6.00,-1.75){};
        \node[NodeST](1)at(1.00,-3.50){\begin{math}\Product\end{math}};
        \node[NodeST](3)at(3.00,-1.75){\begin{math}\Product\end{math}};
        \node[NodeST](5)at(5.00,0.00){\begin{math}\Product\end{math}};
        \draw[Edge](0)--(1);
        \draw[Edge](1)--(3);
        \draw[Edge](2)--(1);
        \draw[Edge](3)--(5);
        \draw[Edge](4)--(3);
        \draw[Edge](6)--(5);
        \node(r)at(5.00,1.5){};
        \draw[Edge](r)--(5);
    \end{tikzpicture}
    \enspace \Rew \enspace
    \begin{tikzpicture}[xscale=.22,yscale=.23,Centering]
        \node(0)at(0.00,-1.75){};
        \node(2)at(2.00,-3.50){};
        \node(4)at(4.00,-5.25){};
        \node(6)at(6.00,-5.25){};
        \node[NodeST](1)at(1.00,0.00){\begin{math}\Product\end{math}};
        \node[NodeST](3)at(3.00,-1.75){\begin{math}\Product\end{math}};
        \node[NodeST](5)at(5.00,-3.50){\begin{math}\Product\end{math}};
        \draw[Edge](0)--(1);
        \draw[Edge](2)--(3);
        \draw[Edge](3)--(1);
        \draw[Edge](4)--(5);
        \draw[Edge](5)--(3);
        \draw[Edge](6)--(5);
        \node(r)at(1.00,1.5){};
        \draw[Edge](r)--(1);
    \end{tikzpicture}\,.
\end{equation}
This rewrite rule is compatible with the lexicographic order on prefix 
words presented at the beginning of Section~\ref{sec:operad_Mag} in the 
sense that the prefix word of the left member of~\eqref{equ:rew_1} is 
lexicographically greater than the prefix word of the right one.

However, the rewrite relation $\RewContext$ induced by $\Rew$ is not
confluent. Indeed, we have
\begin{equation} \label{equ:branching_pair_CAs_3}
    \begin{tikzpicture}[xscale=.22,yscale=.22,Centering]
        \node(0)at(0.00,-7.20){};
        \node(2)at(2.00,-7.20){};
        \node(4)at(4.00,-5.40){};
        \node(6)at(6.00,-3.60){};
        \node(8)at(8.00,-1.80){};
        \node[NodeST](1)at(1.00,-5.40){\begin{math}\Product\end{math}};
        \node[NodeST](3)at(3.00,-3.60){\begin{math}\Product\end{math}};
        \node[NodeST](5)at(5.00,-1.80){\begin{math}\Product\end{math}};
        \node[NodeST](7)at(7.00,0.00){\begin{math}\Product\end{math}};
        \draw[Edge](0)--(1);
        \draw[Edge](1)--(3);
        \draw[Edge](2)--(1);
        \draw[Edge](3)--(5);
        \draw[Edge](4)--(3);
        \draw[Edge](5)--(7);
        \draw[Edge](6)--(5);
        \draw[Edge](8)--(7);
        \node(r)at(7.00,1.35){};
        \draw[Edge](r)--(7);
    \end{tikzpicture}
    \enspace \RewContext \enspace
    \begin{tikzpicture}[xscale=.22,yscale=.20,Centering]
        \node(0)at(0.00,-4.50){};
        \node(2)at(2.00,-4.50){};
        \node(4)at(4.00,-4.50){};
        \node(6)at(6.00,-6.75){};
        \node(8)at(8.00,-6.75){};
        \node[NodeST](1)at(1.00,-2.25){\begin{math}\Product\end{math}};
        \node[NodeST](3)at(3.00,0.00){\begin{math}\Product\end{math}};
        \node[NodeST](5)at(5.00,-2.25){\begin{math}\Product\end{math}};
        \node[NodeST](7)at(7.00,-4.50){\begin{math}\Product\end{math}};
        \draw[Edge](0)--(1);
        \draw[Edge](1)--(3);
        \draw[Edge](2)--(1);
        \draw[Edge](4)--(5);
        \draw[Edge](5)--(3);
        \draw[Edge](6)--(7);
        \draw[Edge](7)--(5);
        \draw[Edge](8)--(7);
        \node(r)at(3.00,1.74){};
        \draw[Edge](r)--(3);
    \end{tikzpicture}
    \qquad \mbox{and} \qquad
    \begin{tikzpicture}[xscale=.22,yscale=.22,Centering]
        \node(0)at(0.00,-7.20){};
        \node(2)at(2.00,-7.20){};
        \node(4)at(4.00,-5.40){};
        \node(6)at(6.00,-3.60){};
        \node(8)at(8.00,-1.80){};
        \node[NodeST](1)at(1.00,-5.40){\begin{math}\Product\end{math}};
        \node[NodeST](3)at(3.00,-3.60){\begin{math}\Product\end{math}};
        \node[NodeST](5)at(5.00,-1.80){\begin{math}\Product\end{math}};
        \node[NodeST](7)at(7.00,0.00){\begin{math}\Product\end{math}};
        \draw[Edge](0)--(1);
        \draw[Edge](1)--(3);
        \draw[Edge](2)--(1);
        \draw[Edge](3)--(5);
        \draw[Edge](4)--(3);
        \draw[Edge](5)--(7);
        \draw[Edge](6)--(5);
        \draw[Edge](8)--(7);
        \node(r)at(7.00,1.35){};
        \draw[Edge](r)--(7);
    \end{tikzpicture}
    \enspace \RewContext \enspace
    \begin{tikzpicture}[xscale=.22,yscale=.22,Centering]
        \node(0)at(0.00,-3.60){};
        \node(2)at(2.00,-5.40){};
        \node(4)at(4.00,-7.20){};
        \node(6)at(6.00,-7.20){};
        \node(8)at(8.00,-1.80){};
        \node[NodeST](1)at(1.00,-1.80){\begin{math}\Product\end{math}};
        \node[NodeST](3)at(3.00,-3.60){\begin{math}\Product\end{math}};
        \node[NodeST](5)at(5.00,-5.40){\begin{math}\Product\end{math}};
        \node[NodeST](7)at(7.00,0.00){\begin{math}\Product\end{math}};
        \draw[Edge](0)--(1);
        \draw[Edge](1)--(7);
        \draw[Edge](2)--(3);
        \draw[Edge](3)--(1);
        \draw[Edge](4)--(5);
        \draw[Edge](5)--(3);
        \draw[Edge](6)--(5);
        \draw[Edge](8)--(7);
        \node(r)at(7.00,1.5){};
        \draw[Edge](r)--(7);
    \end{tikzpicture}\,,
\end{equation}
and the two right members of~\eqref{equ:branching_pair_CAs_3} form a
branching pair which is not joinable.

In order to transform the rewrite relation induced by~\eqref{equ:rew_1} 
into a convergent one, we apply the Buchberger algorithm for
operads~\cite[Section 3.7]{DK10} with respect to the lexicographic order
on prefix words. Following this algorithm, we need to put the right
members of~\eqref{equ:branching_pair_CAs_3} in relation by $\Rew$. To
respect the lexicographic property of the prefix words, this leads to
the new relation
\begin{equation} \label{equ:rew_2}
    \begin{tikzpicture}[xscale=.22,yscale=.22,Centering]
        \node(0)at(0.00,-3.60){};
        \node(2)at(2.00,-5.40){};
        \node(4)at(4.00,-7.20){};
        \node(6)at(6.00,-7.20){};
        \node(8)at(8.00,-1.80){};
        \node[NodeST](1)at(1.00,-1.80){\begin{math}\Product\end{math}};
        \node[NodeST](3)at(3.00,-3.60){\begin{math}\Product\end{math}};
        \node[NodeST](5)at(5.00,-5.40){\begin{math}\Product\end{math}};
        \node[NodeST](7)at(7.00,0.00){\begin{math}\Product\end{math}};
        \draw[Edge](0)--(1);
        \draw[Edge](1)--(7);
        \draw[Edge](2)--(3);
        \draw[Edge](3)--(1);
        \draw[Edge](4)--(5);
        \draw[Edge](5)--(3);
        \draw[Edge](6)--(5);
        \draw[Edge](8)--(7);
        \node(r)at(7.00,1.5){};
        \draw[Edge](r)--(7);
    \end{tikzpicture}
    \enspace \Rew \enspace
    \begin{tikzpicture}[xscale=.22,yscale=.20,Centering]
        \node(0)at(0.00,-4.50){};
        \node(2)at(2.00,-4.50){};
        \node(4)at(4.00,-4.50){};
        \node(6)at(6.00,-6.75){};
        \node(8)at(8.00,-6.75){};
        \node[NodeST](1)at(1.00,-2.25){\begin{math}\Product\end{math}};
        \node[NodeST](3)at(3.00,0.00){\begin{math}\Product\end{math}};
        \node[NodeST](5)at(5.00,-2.25){\begin{math}\Product\end{math}};
        \node[NodeST](7)at(7.00,-4.50){\begin{math}\Product\end{math}};
        \draw[Edge](0)--(1);
        \draw[Edge](1)--(3);
        \draw[Edge](2)--(1);
        \draw[Edge](4)--(5);
        \draw[Edge](5)--(3);
        \draw[Edge](6)--(7);
        \draw[Edge](7)--(5);
        \draw[Edge](8)--(7);
        \node(r)at(3.00,1.74){};
        \draw[Edge](r)--(3);
    \end{tikzpicture}\,.
\end{equation}
The Buchberger algorithm applied on binary trees of degrees $5$, $6$,
and $7$ provides the new relations \\
\input{Buchberger_relations_5-7}

\noindent
We claim that the rewrite relation $\RewContext$ induced by
rewrite rule $\Rew$ satisfying~\eqref{equ:rew_1}, \eqref{equ:rew_2}, 
\eqref{equ:rew_3}---\eqref{equ:rew_11} is convergent. First, for every 
relation $\Tfr \Rew \Tfr'$, we have
$\PrefixWord(\Tfr) > \PrefixWord(\Tfr')$. Therefore, by
Lemma~\ref{lem:prefix_word_termination}, $\RewContext$ is terminating.
Moreover, the greatest degree of a tree appearing in $\Rew$ is~$7$ so
that, from Lemma~\ref{lem:degree_confluence}, to show that $\RewContext$
is convergent, it is enough to prove that each tree of degree at most
$13$ admits exactly one normal form. Equivalently, this amounts to
show that the number of normal forms of trees of arity $n$ is equal
to $\#\CAs{3}(n)$. By computer exploration, we get the same sequence
\begin{equation} \label{equ:dimensions_CAs_3}
    1, 1, 2, 4, 8, 14, 20, 19, 16, 14, 14, 15, 16, 17
\end{equation}
for $\#\CAs{3}(n)$ and for the numbers of normal forms of arity $n$,
when $ 1 \leq n \leq 14$. Hence, we get our following main result.

\begin{Theorem} \label{thm:convergent_rewrite_rule_CAs_3}
    The rewrite rule $\Rew$ satisfying~\eqref{equ:rew_1},
    \eqref{equ:rew_2}, \eqref{equ:rew_3}---\eqref{equ:rew_11} is a
    convergent orientation of the congruence $\CongrCAs{3}$
    of~$\CAs{3}$.
\end{Theorem}

The rewrite rule $\Rew$ has, arity by arity, the cardinalities
\begin{equation}
    0, 0, 0, 1, 1, 2, 3, 4, 0, \dots~.
\end{equation}
We obtain from Theorem~\ref{thm:convergent_rewrite_rule_CAs_3} also
the following consequences.

\begin{Proposition} \label{prop:PBW_basis_CAs_3}
    The set of the trees avoiding as subtrees the ones appearing as
    left members of $\Rew$ is a Poincaré-Birkhoff-Witt basis
    of~$\CAs{3}$.
\end{Proposition}

From Proposition~\ref{prop:PBW_basis_CAs_3}, and by using a result
of~\cite{Gir18} describing a system of equations for the generating
series of syntax trees avoiding some sets of subtrees, we obtain the
following result.

\begin{Proposition} \label{prop:Hilbert_series_CAs_3}
    The Hilbert series of $\CAs{3}$ is
    \begin{equation} \label{equ:Hilbert_series_CAs_3}
        \HilbertSeries_{\CAs{3}}(t) = \frac{t}{(1 - t)^2}
        \left(1 - t + t^2 + t^3 + 2t^4 + 2t^5 - 7t^7 - 2t^8 + t^9 +
        2t^{10} + t^{11}\right).
    \end{equation}
\end{Proposition}

For $n \leq 10$, the dimensions of $\CAs{3}(n)$ are provided by
Sequence~\eqref{equ:dimensions_CAs_3} and for all $n \geq 11$, the
Taylor expansion of~\eqref{equ:Hilbert_series_CAs_3} shows that
$\# \CAs{3}(n) = n + 3$.

\section*{Perspectives}
Our first axis of perspectives consists in collecting properties about
the operads $\CAs{d}$. A natural question consists in finding all the
morphisms between the operads $\CAs{d}$. Some surjective morphisms are
described by Proposition~\ref{prop:quotients_CAs_d} and we can hope to a
full description of these, as well as some possible injections.
Moreover, we can try to obtain a convergent orientation of
$\CongrCAs{d}$ and general expressions of the Hilbert series of
$\CAs{d}$ when $d \geq 4$. By computer exploration, we have the sequence
\begin{equation}
    1, 1, 2, 5, 13, 35, 96, 264, 724, 1973, 5355, 14390
\end{equation}
for the first dimensions for $\CAs{4}$. By applying the Buchberger
algorithm on trees of degrees until $10$, we obtain that a convergent
orientation of $\CongrCAs{4}$ has, arity by arity, the sequence
\begin{math}
    0, 0, 0, 0, 1, 1, 0, 3, 4, 5, 18, 22
\end{math}
for its first cardinalities. Moreover, for $\CAs{5}$, we get the
sequence
\begin{equation}
    1, 1, 2, 5, 14, 41, 124, 384, 1210, 3861, 12440
\end{equation}
of dimensions and the first cardinalities
\begin{math}
    0, 0, 0, 0, 0, 1, 1, 0, 0, 4, 5
\end{math}
for any convergent orientation of $\CongrCAs{5}$. Finally, for
$\CAs{6}$, we get the sequence
\begin{equation}
    1, 1, 2, 5, 14, 42, 131, 420, 1375, 4576, 15431
\end{equation}
of dimensions and the first cardinalities
\begin{math}
    0, 0, 0, 0, 0, 0, 1, 1, 0, 0, 0
\end{math}
for any convergent orientation of $\CongrCAs{6}$. We can notice that
only $\CAs{3}$ seems to have oscillating first dimensions.

A second axis concerns a complete understanding of $\CAs{3}$. We can
try to construct an explicit basis of this operad.
Proposition~\ref{prop:PBW_basis_CAs_3} describes a basis in terms of
trees avoiding some patterns but, we can hope to find a simpler
description. This includes the description of a family of combinatorial
objects forming a basis of $\CAs{3}$ and an adequate definition of a
partial composition map $\circ_i$ on these. Moreover, a natural
question is to explore the suboperads $\CAs{3}$ in the category of
vector spaces.

In a last axis, we can consider further generalizations of $\As$ being
quotients of $\Mag$ by congruences defined by identifying certain binary
trees of a same fixed degree. A possible question is, as presented in
the introduction, to investigate if combinatorial properties of the
trees belonging to a same equivalence class imply algebraic properties
on the obtained operads.

\bibliographystyle{plain}
\bibliography{Bibliography}

\end{document}